\documentclass[11pt]{article}

 \usepackage{amsfonts,amssymb,graphicx,amsmath, amsthm}
\usepackage[colorlinks=true,linkcolor=blue,citecolor=blue]{hyperref}

\usepackage{multicol}

\newtheorem{rema}{Remark}

\newtheorem{lemm}{Lemma}
\newtheorem{theo}{Theorem}

\newcommand{\R}[1][]{\ensuremath{{\mathbb{R}^{#1}} }}
\renewcommand{\S}[1][]{\ensuremath{{\mathbb{S}^{#1}} }}
\renewcommand{\H}[1][]{\ensuremath{{\mathbb{H}^{#1}} }}

\newcommand{\M}{{\cal M}}

\newcommand{\s}{{\cal S}}
\newcommand{\<}{\langle}
\renewcommand{\>}{\rangle}
\newcommand{\ga}{\gamma}

\newcommand{\al}{\alpha}
\newcommand{\eps}{\epsilon}

\newcommand{\ka}{\kappa}

\newcommand{\be}{\beta}
\newcommand{\si}{\sigma}

\date{}

\title{Marginally trapped submanifolds in Lorentzian space forms and in the Lorentzian product of a space form by the real line}
\author{ Henri Anciaux\footnote{partially supported by CNPq (PQ 302584/2007-2) and Fapesp (2010/18752-0)}, Yamile Godoy\footnote{FaMAF-CIEM; partially supported by CONICET, FONCyT, SECyT (UNC)}}

\begin{document}

\maketitle

\centerline{\textbf {\large{Abstract}}}

\bigskip

{\small We give local, explicit representation formulas for $n$-dimensional spacelike submanifolds which are marginally trapped 
in the Minkowski space $\R^{n+2}_1,$ the de Sitter space $d\S^{n+2},$ the anti de Sitter space $Ad\S^{n+2}$ and the Lorentzian 
products $\S^{n+1} \times \R$ and $\H^{n+1} \times \R$ of the sphere and the hyperbolic space  by the real line.}

\bigskip

\centerline{\small \em 2000 MSC: 53A10,  53C42
\em }


\section*{Introduction}

Let $\s$ be a  submanifold of a pseudo-Riemannian manifold $({\cal N},g).$ If the induced metric on $\s$ is non-degenerate, one may define its mean curvature $\vec{H},$ a normal vector field along $\s.$ We shall say that $\s$ is \em marginally trapped \em if $\vec{H}$ is a null vector, i.e.\ $g(\vec{H},\vec{H})$ vanishes identically.\footnote{Some authors call such submanifolds \em quasi-minimal \em or \em pseudo-minimal. \em} Of course this may happen only if $\s$ has codimension greater than two and if the induced metric on the normal bundle is indefinite.

The case of a spacelike surface $\s$ of a four-dimensional Lorentzian manifold $({\cal N},g)$ is the most interesting  because of its physical interpretation  in the setting of General Relativity: marginally outer trapped surfaces (MOTS) play a fundamental role in the study of black holes and spacetime singularities (\cite{Pe},\cite{CGP}).
Despite their physical relevance and the fact that marginally trapped is the most natural curvature condition which is purely pseudo-Riemannian,
these  submanifolds
 are still not very well understood.  After a seminal paper (\cite{Ro}) where marginally trapped submanifolds are called \em semi-minimal, \em there have been recent work on the classification of
 marginally trapped surfaces satisfying
several additional properties, such as being Lagrangian (\cite{CD}), isotropic (\cite{CFG}), having flat normal bundle (\cite{AGM}), constant curvature (\cite{Ch}) or  positive relative nullity (\cite{CVdV},\cite{VdV}). On the other hand, in \cite{Pa} (see also \cite{APa}) a very interesting minimization property has been discovered concerning
marginally trapped surfaces: a minimal spacelike surface of Minkowski space $\R^4_1$, although it is unstable,  minimizes the area among  marginally trapped surfaces satisfying a natural boundary data.

The purpose of this paper is to give local, explicit representation formulas for $n$-dimensional marginally trapped submanifolds 
in some of the simplest Lorentzian manifolds: the 
 Minkowski space $\R^{n+2}_1,$ the non-flat Lorentzian spaces forms, i.e.\ the de Sitter space $d\S^{n+2}$ and the anti de Sitter space $Ad\S^{n+2},$ and finally the Lorentzian 
products $\S^{n+1} \times \R$ and  $\H^{n+1} \times \R$ of the sphere and the hyperbolic space by the real line.

Our construction is inspired by \cite{An}  and is based, although not explicitely, on the contact structure enjoyed by the space of null geodesics of a pseudo-Riemannian manifold (\cite{GS},\cite{KT}).  For example, in the case of Minkowski space $\R^{n+2}_1,$ a spacelike, $n$-dimensional submanifold $\bar{\s}$ is locally described in terms of its height function (i.e.\ its timelike coordinate) and a hypersurface $\s$
of $\R^{n+1}$. Then the marginally trapped condition amounts to a simple algebraic relation between the height function of $\bar{\s}$ and the second fundamental form of $\s.$ The interpretation in terms of contact geometry is the following:  the set of null geodesics normal to $\s$ is a Legendrian submanifold in
 the set of null geodesics of $\R^{n+2}_1,$ which is contactomorphic to the unit tangent bundle of $\R^{n+1}$; then
 the hypersurface $\s$ is nothing but the projection on the basis $\R^{n+1}$ of this Legendrian submanifold.

This idea works in the same way in the other simple Lorentzian spaces $d\S^{n+2}, Ad\S^{n+2}$, $\S^{n+1} \times \R$ and  $\H^{n+1} \times \R$. The construction can be performed in Robertson-Walker spaces as well, but the analysis becomes quite more involved since the equation relating the height function of $\bar{\s}$ and the second fundamental form of $\s$ is not any more polynomial, but remains algebraic. This case is discussed in \cite{AC}.

  The authors would like to thank Marcos Salvai for valuable suggestions. 

\section{Statement of results}
Let $(\R^{n+1},\<.,.\>_0)$ be the Euclidean space endowed with its canonical Riemannian metric 
$$\<.,.\>_0:= dx_1^2 + ... + dx_{n+1}^2,$$ and denote by 
$$\<.,.\>_1:=\<.,.\>_0 - dx_{n+2}^2$$
 the flat Lorentzian metric of the Cartesian product $\R^{n+2}_1 = \R^{n+1}\times \R.$

We denote by $\iota: \R^n \to \R^{n+1}$ 
 the canonical totally geodesic embedding $\iota(x)=(x,0)$
and denote by $\nu_0=(0,...,0,1)$ its (constant) unit normal vector. 

We recall that the second fundamental form $h$ of an immersion $\M \to ({\cal N},g)$ with non degenerate first fundamental form is the symmetric tensor
$h :T\M \times T\M \to N \M$
defined by $h(X,Y):= (D_X Y)^{\perp}$, where $(.)^\perp$ denotes the projection onto the normal space $N\M$ and $D$ is the Levi-Civita connection
of $g.$ If $\nu$ is a normal vector field along $\M$, we have the following important relation: $g(h(X,Y), \nu) =- g(D_X \nu, Y).$ The mean curvature vector
of the immersion is the trace of $h$ with respect to the induced metric.

\begin{theo} \label{one} 
Let $\Omega$ be an open domain of $\R^n$ and $\tau \in C^2(\Omega)$.
Then the immersion 
$\bar{\varphi} : \Omega \to \R^{n+2}_1$  defined by
$$\bar{\varphi}(x)=  \big(\iota(x),0 \big) + \tau(x) \big(\nu_0,1 \big)$$
 is flat  and its
  second fundamental form  is  given by
$$\bar{h} (X,Y)= Hess_{\tau}(X,Y) (\nu_0,1).$$
In particular $\bar{\varphi}$ has null second fundamental form and is therefore marginally trapped.
Conversely, any $n$-dimensional spacelike submanifold with null second fundamental form is locally congruent to the image of such an immersion.

\smallskip

 Let $\varphi$ be an immersion of class $C^4$ of an $n$-dimensional manifold $\M$ into  $\R^{n+1}$ and denote by $ \nu$ the Gauss map of $\varphi$, which is therefore $\S^{n}$-valued.
Assume that $\varphi$ admits  $p$ distinct, non-vanishing principal curvatures $ \ka_1 , ..., \ka_p, \, p \geq 2$ with multiplicity $m_i$ and denote by $\tau_i$  the $p-1$ roots of the polynomial
 $$ P(\tau):=\sum_{i=1}^{p} m_i \prod_{j \neq i}^{p}(\ka_j^{-1}- \tau).$$
Then the  $p-1$  immersions 
$\bar{\varphi}_i : \M \to \R^{n+2}_1$ defined by
$$\bar{\varphi}_i= (\varphi + \tau_i  \nu ,\tau_i)$$  
are marginally trapped.

Conversely, any $n$-dimensional marginally trapped submanifold of  $\R^ {n+2}_1$ whose second fundamental form is not null  is locally congruent to the image of 
 such an immersion.

\smallskip

In particular, in the $n=2$ case, let   $\varphi$ be a non-flat $C^4$-immersion  of a surface $\M$ into $\R^3$ 
which is  free of umbilic points. Denote by $\nu$ its Gauss map, by $H$ and $K$  the mean curvature with respect to $ \nu$ and the Gaussian curvature of $\varphi.$ Then immersion $\bar{\varphi}: \M \to \R^4_1$ defined by
$$  \bar{\varphi}:=\left(\varphi + \frac{H}{K}  \nu ,  \frac{H}{K} \right),$$
is  marginally trapped.   As a corollary, an immersed surface which is contained in a time slice $\{x_4=const.\}$ is marginally trapped if and only if $\frac{H}{K}$ is constant, i.e.\  $\varphi$ is linear Weingarten.

\end{theo}

It may be interesting to relate the latter formula to one found by B. Palmer: it is proved in \cite{Pa} that, given $\Omega$ an open subset of $\S^2$
and $f \in C^4(\Omega),$  the immersion $\varphi: \Omega \to \R^4_1$ defined by
\begin{equation} \label{zero}
\bar{\varphi}(x) = \left( \nabla f (x)  + f(x) x , -f(x) - \frac{1}{2} \Delta f(x) \right),
\end{equation}
where $\nabla$ and $\Delta$ denote the gradient and Laplace operators w.r.t.\ the round metric,
 is marginally trapped.

We first observe that a convex surface of Euclidean space  may be constructed from its support function: let
$\varphi : \Omega \to \R^3$ parametrized by its   unit normal vector $\nu$  and introduce $f(\nu)=\< \varphi(\nu),\nu \>_0 ,$ i.e.\ $f$ is the support function of $\varphi.$ Hence we have  (see \cite{AG}): 
\begin{equation} \label{uno}
\varphi(\nu)=f(\nu)\nu + \nabla f(\nu).
\end{equation}
It will be seen along the proof of Theorem 1 that a spacelike immersion $\bar{\varphi}: \Omega \to \R^4_1$ such that 
the null geodesic $\{ (\varphi(\nu),0) + t(\nu,1) | \, t \in \R \}$ crosses orthogonally the surface  $\bar{\varphi}(\Omega)$ at $\bar{\varphi}(\nu)$ takes the form
$$ \bar{\varphi}(\nu) = \left( \varphi(\nu) +\tau(\nu) \nu, \tau(\nu)  \right),$$
where $\tau$ is a smooth real map on $\Omega.$ Moreover, the mean curvature vector of $\bar{\varphi}$ is collinear to the null vector field $\bar{ \nu}=(\nu,1)$ if and only if 
$\tau= \frac{H}{K}.$
On the other hand, the following formula holds (see \cite{AG}):
\begin{equation} \label{dos}
\frac{H}{K} = -f- \frac{1}{2} \Delta f,\end{equation} 
so using Formulas (\ref{uno}) and (\ref{dos}) together with the formula $ \bar{\varphi}:=\left(\varphi + \frac{H}{K}  \nu ,  \frac{H}{K} \right)$ of Theorem \ref{one},
 we recover Palmer's formula (Formula (\ref{zero})).

\medskip

The same construction may be applied in the case of the Lorentzian  space forms. In order to state the result we  set
$$\S^{n+1}:= \{  x \in \R^{n+2} |  \,  \<x,x\>_0 =1 \} \quad
\mbox{and}
\quad \H^{n+1}:= \{  x \in \R^{n+2} |  \,  \<x,x\>_1 = -1 \}.$$
These are the hyperquadric models of the $n+1$-dimensional Riemannian space forms.
Similarly  the $(n+2)$-dimensional Lorentzian space forms are defined as follows:
$$d \S^{n+2}:= \{  x \in \R^{n+3} |  \,  \<x,x\>_1 = 1 \} \quad \mbox{and} \quad
 Ad\S^{n+2} := \{ x\in \R^{n+3} |  \,  \<x,x\>_2 = -1 \},$$
where 
$$ \<.,.\>_2 := dx_1^2 + ... + dx_{n+1}^2 -  dx_{n+2}^2 - dx_{n+3}^2 .$$

\begin{theo} \label{unbis}

Let $\Omega$ be an open domain of $\S^n$ (resp.\ $\H^n$) and $\tau \in C^2(\Omega) $.  Denote by $\iota: \S^n \to \S^{n+1}\subset \R^{n+2}$ (resp.\  $\H^n \to \H^{n+1} \subset \R^{n+2}_1$)
 the canonical totally geodesic embedding $\iota(x)=(x,0)$
and denote by $\nu_0=(0,...,0,1)$ the corresponding (constant) unit normal vector. 
Then the immersion 
$\bar{\varphi} : \S^n \to d\S^{n+2}$   (resp.\  $\H^n \to Ad\S^{n+2}$) defined by
$$\bar{\varphi}(x)=  \big(\iota(x),0 \big) + \tau(x)  \big(\nu_0,1 \big)$$
 is flat and its second fundamental form  is  given by
$$\bar{h} (X,Y)= Hess_{\tau}(X,Y) (\nu_0,1).$$
In particular $\bar{\varphi}$ has null second fundamental form and is therefore marginally trapped. 
Conversely, any $n$-dimensional spacelike submanifold with null second fundamental form is locally congruent to the image of such an immersion.

\smallskip

 Let $\varphi$ be an immersion of class $C^4$ of an $n$-dimensional manifold $\M$ into  $\S^{n+1}$ (resp.\  $\H^{n+1}$) and denote by $ \nu$ the Gauss map of $\varphi$, which is therefore $\S^{n+1}$-valued (resp.\ $d\S^{n+1}$-valued).
Assume that $\varphi$ admits  $p$ distinct, non-vanishing principal curvatures $ \ka_1 , ..., \ka_p, \, p \geq 2$ with multiplicity $m_i$ and denote by $\tau_i$  the $p-1$ roots of the polynomial
 $$ P(\tau):=\sum_{i=1}^{p} m_i \prod_{j \neq i}^{p}(\ka_j^{-1}- \tau).$$
Then the  $p-1$  immersions 
$\bar{\varphi}_i : \M \to d\S^{n+2}$ (resp.\ $Ad\S^{n+2}$) defined by
$$\bar{\varphi}_i= (\varphi + \tau_i  \nu ,\tau_i)$$  
are marginally trapped.

Conversely, any $n$-dimensional marginally  trapped submanifold of $d\S^{n+2}$ (resp.\ $of Ad\S^{n+2}$) whose second fundamental form is not null  is locally congruent to the image of 
 such an immersion.

\end{theo}

We observe that all the examples found in \cite{CVdV} and \cite{CFG} have null fundamental form (see Section \ref{examples}).

\medskip

The construction works as well  in the case of the Lorentzian product of a space form by the real line. We endow $\S^{n+1} \times \R$  with the Lorentzian metric $ \<.,.\>_0 -dx_{n+3}^2,$ where $\<.,.\>_0$ is the round metric of  $\S^{n+1}$ and $x_{n+3}$ denotes the canonical coordinate of the real line $\R.$ 
\begin{theo} \label{two} 
There is no non-totally geodesic $n$-dimensional submanifold of $\S^{n+1} \times \R$ with null second fundamental form.

\smallskip

Let $\varphi$ be an immersion of class $C^4$  of an $n$-dimensional manifold $\M$ into  $ \S^{n+1}.$ 
 Denote by $ \nu$ the Gauss map of $\varphi$ and by $\ka_1, ...,\ka_p$ its $p$ distinct curvatures with multiplicity $m_i.$ Then the
 polynomial
$$ P(s):=\sum_{i=1}^{p}m_i (\kappa_i s+1)\prod_{j\neq i}^{p}(s-\kappa_j)$$
admits exactly $p-1$ roots $s_i$ if $\varphi$ is minimal and $p$ roots otherwise. Moreover, the $p-1$ or $p$
  immersions $\bar{\varphi}_i: \M \to \S^{n+1} \times \R$ defined by
$$  \bar{\varphi}_i:=\left( \frac{s_i \varphi +\nu}{\sqrt{1+s_i^2}} ,\cot^{-1} s_i \right), \, 1 \leq i \leq p-1 \mbox { or } p,$$
are marginally trapped. 

Conversely, any $n$-dimensional marginally trapped submanifold of $\S^{n+1} \times \R$ is locally congruent to the image of such an immersion.

\smallskip

In particular, in the $n=2$ case, given a  non minimal $C^4$-immersion $\varphi$ of a surface into $\S^3$ and
$$ a:=\frac{\kappa_1 \kappa_2-1}{\ka_1+\ka_2},$$
the two immersions into $\S^{3} \times \R$ defined by
$$\bar{\varphi}_{\pm}:=\left( \frac{ (a\pm \sqrt{a^2+1})\varphi + \nu}{\sqrt{2}\sqrt{a^2+1 \pm a\sqrt{a^2+1}}} , \cot^{-1} (a \pm \sqrt{a^2+1}) \right)$$
are marginally trapped. 

\end{theo}

Analogously, $\H^{n+1} \times \R$ is endowed with the metric $ \<.,.\>_1 -dx_{n+3}^2,$ where $\<.,.\>_1$ is the standard metric of  $\H^{n+1}$ and  $x_{n+3}$ denotes the canonical coordinate of the real line $\R.$

\begin{theo} \label{twobis} 
There is no non-totally geodesic $n$-dimensional submanifold of $\H^{n+1} \times \R$ with null second fundamental form.

\smallskip

Let $\varphi$ be an immersion of class $C^4$  of an $n$-dimensional manifold $\M$ into  $ \H^{n+1}.$ 
 Denote by $ \nu$ the Gauss map of $\varphi$ and by $\ka_1, ...,\ka_p$ its $p$ distinct curvatures with multiplicity $m_i.$ 
Denote by $s_i$, $1 \leq i \leq q \leq p,$ the $q$ roots of the polynomial
$$ P(s):=\sum_{i=1}^{p}m_i (  \ka_i s -1)\prod_{j\neq i}^{p}(s-\kappa_j )$$
satisfying $|s_i| >1.$
Then the $q$
  immersions $\bar{\varphi}_i: \M \to \H^{n+1} \times \R$ defined by
$$  \bar{\varphi}_i:=\left( \frac{s_i \varphi +\nu}{\sqrt{s_i^2-1}} ,\coth^{-1} s_i \right), \, 1 \leq i \leq q,$$
are marginally trapped. 

Conversely, any $n$-dimensional marginally trapped submanifold of $\H^{n+1} \times \R$ is locally congruent to the image of such an immersion.

\smallskip

In particular, in the $n=2$ case, given a  non minimal $C^4$-immersion $\varphi$ of a surface into $\H^3,$
such that $a:=\frac{\ka_1 \ka_2 +1}{\ka_1 +\ka_2} \in (1,\infty),$
the immersion into $\H^{3} \times \R$ defined by
$$\bar{\varphi}:=\left( \frac{ (a +  \sqrt{a^2-1})\varphi + \nu}{\sqrt{2}\sqrt{a^2 -1 + a\sqrt{a^2-1}}} , \frac{1}{2}\coth^{-1} (a) \right)$$
is marginally trapped.

\end{theo}


\section{The Minkowski case: proof of Theorem \ref{one}} \label{mink}

Let $\bar{\varphi}=(\psi,\tau)$ be an immersion of a $n$-dimensional manifold $\M$ into $ \R^{n+2}_1 $ which is spacelike, i.e.\ the induced metric
$\bar{g}:=\bar{\varphi}^*\<.,.\>_1$ is definite positive.  In particular the induced metric on the normal space of  $\bar{\varphi}$ is Lorentzian and we may define locally two null, non-vanishing normal vector fields. Moreover a  null  vector field $\bar{\nu}$ may be normalized on the following form:
$ \bar{\nu}=(\nu,1),$ with $\nu :  \M  \to \S^{n}.$ 
From now on we consider a null normal vector field $\bar{\nu}:=(\nu,1)$ 
 and we set   $\varphi:=\psi - \tau \nu.$

\begin{lemm} \label{l1mink}
 The map $(\varphi, \nu) : \M \to \R^{n+1} \times \S^n$ is an immersion. 
\end{lemm}

\begin{proof}
Suppose $(\varphi,\nu)$ is not an immersion, so that there exists a non-vanishing vector $v \in T\M$ such that $(d\varphi(v),d\nu(v))=(0,0).$
Since we have $d\psi=d\varphi + \tau d\nu + d\tau \nu,$
it follows that 
$$d \bar{\varphi}(v) =(d\psi(v), d\tau(v))= (d\tau(v) \nu , d\tau(v))= d\tau(v) \bar{\nu},$$
which is a null vector. This contradicts the assumption that $\bar{\varphi}$ is spacelike.
\end{proof}

\begin{lemm} \label{l2} We have the following relation\footnote{This corresponds to the fact that the immersion $(\varphi,\nu)$ is \em Legendrian \em with respect to the canonical contact structure of the unit bundle of $\R^{n+1}$.}:
$$\<d\varphi, \nu\>_0=0.$$
\end{lemm}

\begin{proof} Using again that $d\psi= d\varphi + \tau d\nu + d\tau \nu$ and observing 
 that $\<\nu, d\nu\>_0=0,$ we have 
$$ 0 =  \<d\bar{\varphi},\bar{\nu}\>_1 = \<(d\psi, d\tau),(\nu, 1)\>_1= \<d\psi , \nu \>_0 - d\tau= \<d\varphi, \nu\>_0$$

\end{proof}

 \begin{lemm}  \label{l3mink}
Given $x \in \M$ and $ \eps >0,$ there exists a neighbourhood $U$ of $x$ and $t_0 \in (-\eps, \eps)$ such that $\varphi + t_0 \nu$ is an immersion of $U$, and $\nu\big|_U$ is its Gauss map.
  \end{lemm}
 
\begin{proof}

The claim follows from the fact that, $\forall x \in \M$,  the set
$$ \{ t \in \R | \,  \, d\varphi_x + t \, d\nu_x \mbox {  has not maximal rank} \}$$
 contains at most $n$ elements.
To see this, observe that given a pair of distinct real numbers $(t,t')$, we have
$$Ker (d\varphi_x + t \, d \nu_x) \cap Ker  (d\varphi_x + t' \, d \nu_x) = \{ 0\}$$ 
(otherwise we would have a contradiction with the fact that $(\varphi,\nu)$ is an immersion).
Hence there cannot be more than $n$ distinct values $t$ such that $Ker (d\varphi_x + t \,  d \nu_x) \neq \{ 0\}.$ Moreover such real numbers $t$ depend continuously on the point $x \in \M$, so we may choose a  neighbourhood $U$ of $x$ such that
 $ \{ t \in \R \, | \,  \,    \varphi +  t \nu \mbox { is  an immersion of } \, U \}$ contains a neighbourhood of $0$, which implies the first part of the claim.

The fact that $\nu$ is the Gauss map of $\varphi + t_0 \nu$ comes from Lemma \ref{l2}:
$$\<d(\varphi + t_0 \nu),\nu\>_0=\< d\varphi,\nu\>_0 + t_0\<d \nu,\nu\>_0=0.$$
\end{proof}

 Since the whole discussion is local,  Lemma \ref{l3mink} shows that there is no loss of generality in assuming that $\varphi$ is an immersion: if it is not the case, we may
 translate the immersion $\bar{\varphi}$ along the vertical direction, setting $\bar{\varphi}_{t_0}:=\bar{\varphi} - (0,t_0).$
Of course $\bar{\varphi}$ is marginally trapped if and only if $\bar{\varphi}_{t_0}$ is so, and moreover
the vector field $\bar{\nu}$ is still normal to $\bar{\varphi}_{t_0}.$
Finally, observe that the map $\varphi_{t_0}: \M \to \R^{n+1}$ associated to $\bar{\varphi}_{t_0}$ is 
$$\varphi_{t_0} =\psi - (\tau - {t_0}) \nu = \psi - \tau \nu + t_0 \nu = \varphi + t_0 \nu,$$
hence an immersion.

We now describe the first fundamental form of $\bar{\varphi}$ and its second fundamental form with respect to $\bar{ \nu}$, both in terms of the
geometry of the immersion $\varphi$:

\begin{lemm} \label{geo}
Denote by $g:=\varphi^* \<.,.\>_0$ the metric induced on $\M$ by $\varphi$ and $A$ the shape operator associated to $ \nu,$ 
i.e.\ $A (v):=-d\nu(v), \, \forall v \in T\M.$
Then the metric 
$\bar{g}:=\bar{\varphi}^\ast \<.,.\>_1$ induced on  $\M$ by $\bar{\varphi}$ is given by the formula
$$\bar{g}=g(.,.)-2 \tau g(A.,.)+ \tau^2 g(A.,A.).$$
In particular, the non-degeneracy assumption on $\bar{g}$ implies that $\tau^{-1}$ is not equal to any principal curvature of $\varphi$.
Moreover,  the second fundamental form of $\bar{\varphi}$ with respect to $\bar{ \nu}$ is given by
$$\bar{h}_{\bar{ \nu}}:= \<\bar{h}(.,.), \bar{\nu}\>_1 = g(.,A.)- \tau  g(A.,A.).$$
and 
\begin{equation} \label{H} \<\vec{H}_{\bar{\varphi}},\bar{ \nu}\>_1 =  
 \frac{1}{n} \sum_{i=1}^{n} \frac{\ka_i}{(1-\tau \ka_i)}, \end{equation} 
where the $\ka_i$ are the principal curvatures of $\varphi.$ \end{lemm}

\begin{proof}

Since $\<d\varphi, \nu\>_0=\<d \nu, \nu\>_0=0$, we have, given $v_1,v_2 \in T_x\M^n,$
\begin{eqnarray*}
\bar{g}(v_1,v_2)&=&\<d\bar{\varphi}(v_1),d\bar{\varphi}(v_2)\>_1\\
&=&\<d\varphi(v_1),d\varphi(v_2)\>_0 +\tau \<d\varphi(v_1),d \nu(v_2)\>_0+ \tau \<d \nu(v_1),d\varphi(v_2)\>_0 \\
&&+\tau^2 \<d \nu(v_1),d \nu(v_2)\>_0
+ d \tau (v_1)d \tau (v_2)\< \nu, \nu\>_0- d \tau (v_1)d \tau (v_2)\\
&=&g(v_1,v_2)-\tau (g(v_1,Av_2)+g(Av_1,v_2))+ \tau^2 g(Av_1,Av_2)\\
&=& g(v_1,v_2)-2 \tau g(Av_1,v_2)+ \tau^2 g(Av_1,Av_2).
\end{eqnarray*}
We calculate the second fundamental form of $\bar{\varphi}$ with respect to $\bar{ \nu}:=( \nu,1)$:
\begin{eqnarray*} \bar{h}_{\bar{ \nu}}&=& -\< d\bar{\varphi},d\bar{ \nu}\>_1 \\ &=&-\<d\varphi + \tau  d \nu + d\tau  \nu,d \nu\>_0\\ 
&=&-\<d\varphi,d \nu\>_0 - \tau  \<d \nu ,d \nu\>_0\\ &=& g(.,A.) - \tau g(A.,A.).
\end{eqnarray*}
To complete the proof,  observe that in the totally umbilic case $A=\ka Id$, we obviously have
$$\<\vec{H}_{\bar{\varphi}},\bar{ \nu}\>_1=\frac{\ka}{1-\tau \ka}.$$
If $\varphi$ is not totally umbilic, we introduce, away from isolated umbilic points, 
a
principal orthonormal frame $(e_1, ... , e_n)$ along $\M,$ 
i.e.\ such that $g(e_i,e_j)= \delta_{ij}$ and $ Ae_i = \kappa_i e_i.$ Hence
\begin{eqnarray*} \bar{g}(e_i,e_j)&=& (1 - 2\tau  \ka_i+ \tau^2 \ka_i^2) \delta_{ij} \\
\bar{h}(e_i,e_j)&=&\ka_i (1 - \tau \ka_i)\delta_{ij}  
\end{eqnarray*} 
and the proof follows.

\end{proof}

We are now in position to complete the proof of Theorem \ref{one}. We first assume that  $\varphi$ is totally geodesic, i.e.\ $A$ vanishes. This is locally equivalent to assume that $\nu$ is constant, and without loss of generality, we may assume that $\nu=\nu_0:=(0,...,0,1).$

From Equation \ref{H} it is immediately seen  $\bar{h}_{\bar{\nu}}$ vanishes,
so the second fundamental form $\bar{h}$ of $\bar{\varphi}$ takes value in the null line directed by $\bar{\nu}.$ 
It is then straightforward to check that $\bar{h} (X,Y)= Hess_{\tau}(X,Y) (\nu,1).$
 We therefore recover the
first part of Theorem \ref{one}. 

In order to complete the proof we order the non-vanishing  principal curvatures  $\ka_i$, taking into account their  
multiplicity $m_i$, in such a way that  the corresponding radii of curvature   are increasing
$r_1 := \ka_1^ {-1} < ... < r_p:=\ka^{-1}_p$. Hence
\begin{align*}
  & \<\vec{H}_{\bar{\varphi}},\bar{ \nu}\>_1=0\\
 \Longleftrightarrow & \sum_{i=1}^{p} \frac{m_i \ka_i}{(1- \tau \ka_i)}=0\\
 \Longleftrightarrow  & \sum_{i=1}^{p} \frac{m_i}{r_i- \tau}=0 \\
 \Longleftrightarrow &P(\tau):=\sum_{i=1}^{p} m_i \prod_{j \neq i}^{p}(r_j-\tau)=0.
\end{align*}
 We have
$$ P (r_i)= \sum_{k=1}^{p-1} m_k \prod_{j \neq k}^{p-1}(r_j-r_i)=m_i \prod_{j \neq i}^{p-1}(r_j-r_i) .$$
It follows that $P(r_p) > 0,$ $P(r_{p-1}) <0$ and that more generally  the signs of
$ P(r_i)$, $i =1, ..., p$ are alternate. We deduce that $P(\tau)$ admits at least $p-1$  distinct roots $\tau_i$, $i=1, ..., p-1,$ satisfying
 $ r_i < \tau_i < r_{i+1}.$ Since $P(\tau)$ has degree $p-1,$ is has no other roots.

\begin{rema} If $\varphi$ is minimal, $\tau=0$ is a root of $P(\tau)$. The corresponding immersion $\bar{\varphi}=(\varphi,0)$ is not only marginally trapped, but minimal.

\end{rema}


 \section{The de Sitter and anti de Sitter cases:  proof of Theorem \ref{unbis}}

\subsection{The de Sitter case}
Let $\bar{\varphi} =(\psi, {\tau}): \M \to d\S^{n+2}$ an immersion such that the induced metric
$\bar{g}:= \bar{\varphi}^\ast \<.,.\>_1$ is spacelike. Let $\bar{ \nu}=( \nu,1)$ be one of the two normalized, null
normal field to $\bar{\varphi}.$ We define the  \em null projection \em of $\bar{\varphi}$ to be $\varphi:= \psi - {\tau} \nu.$
The fact that $(\nu,1) \in T_{\bar{\varphi}} d\S^{n+2},$ i.e.\
$ 0 = \< (\psi, {\tau}),(\nu, 1)\>_1 = \<\psi, \nu\>_0 -{\tau},$ implies that $\<\psi, \nu\>_0  ={\tau}.$
Hence
\begin{eqnarray*}
\<{\varphi}, {\varphi}\>_0 &=& \<\psi,\psi\>_0 -2 {\tau}\<\psi, \nu\>_0 + {\tau}^2 \<\nu,\nu\>_0 \\
&=& \<\psi,\psi\>_0 -{\tau}^2\\
&=& \<\bar{\varphi},\bar{\varphi}\>_1\\
&=& 1 ,
\end{eqnarray*}
which shows that $\varphi$ is $\S^{n+1}$-valued.
The proofs of the next two lemmas are omitted, since they are similar to the Minkowski case:

\begin{lemm}\label{leg}
 The map $(\varphi, \nu) : \M \to \S^{n+1} \times \S^{n+1}$ is an immersion.
\end{lemm}

\begin{lemm} We have the following relations:\footnote{This corresponds to the fact that the immersion $(\varphi,\nu)$ is \em Legendrian \em with respect to the canonical contact structure of the unit bundle of $\S^{n+1}$.} 
$$ \<\varphi, \nu\>_0=0 \quad \mbox{ and } \quad
\<d\varphi, \nu\>_0=0.$$
\end{lemm}
Unlike in the Minkowski case, there is no vertical translation in $d\S^{n+2}$. 
We may however, up to a arbitrarily small, linear pertubation, assume that $\varphi$ is an immersion.

 \begin{lemm}  Given $x \in \M$ and $ \eps >0,$ there exists a neighbourhood $U$ of $x$,  $\alpha \in (-\eps,\eps)$ and a hyperbolic rotation $R^{\al}$ of angle $\alpha$ such that the null projection $\varphi^{\al}$ of $\bar{\varphi}^{\al} := R^{\al} \bar{\varphi}$ is an immersion.
  \end{lemm}

\begin{proof} 

Set
$$R^{\al}=\left(\begin{array}{ccc} \cosh \al &  & \sinh \al \\  & Id & \\ \sinh \al &&  \cosh \al \end{array} \right) 
 \in SO(n+2,1)$$
and $\bar{\varphi}^\al := R^{\al} \bar{\varphi},$ $\bar{\nu}^\al:= R^{\al} \bar{\nu}.$
 Observe that $\bar{\nu}^{\al}:=(\nu^\al, \si^\al)$ is not anymore normalized, a priori, since its last component $\si^\al:=\bar{\nu}^{\al}_{n+3} $ is
equal to $\cosh (\al) + \sinh (\al) \nu_1,$ where $\nu_1$ is the first component of the vector $\nu.$

Nevertheless, the null geodesic passing through the point $\bar{\varphi}^\al$ and directed by the vector $\bar{\nu}^{\al}$
crosses the slice $d\S^{n+2} \cap \{  x_{n+3}=0\}$ at the point
$$(\varphi^\al,0):= \left(\psi^\al  - \frac{\tau^\al}{\si^\al} \nu^\al,0 \right).$$
Clearly, $\varphi^\al$ is an immersion
 if and only if $R^{-\al}\varphi^{\al}=\psi - \frac{\tau^\al}{\si^\al} \nu= \varphi +\left(\tau - \frac{\tau^\al}{\si^\al}\right)\nu $ is so.
Observe that 
\begin{eqnarray*} 
\tau - \frac{\tau^\al}{\si^\al}&=& \tau - \frac{\cosh (\al) \tau + \sinh (\al) \psi_1}{\cosh (\al)+ \sinh (\al ) \nu_1}\\
&=& \tau - \frac{ \tau + \tanh (\al) \psi_1}{1+ \tanh (\al ) \nu_1}\\
&=& \tanh (\al) (-\psi_1 + \tau \nu_1 ) + o ( \al)\\
&=& -\al \, \varphi_1 + o ( \al).\end{eqnarray*}
Now, assume that  $R^{-\al}\varphi^\al$ fails to be an immersion in any compact neighbourhood $U$ of $x$, $\forall \al \in (-\eps, \eps)$. Hence there exists 
a sequence   $(x_n, v_n) \in T^1 \M$ (the unit tangent bundle of $\M$), such that $x_n \to x$  and
$ d (R^{-1/n}\varphi^{1/n})_{x_n}(v_n)=~0$.
We have
\begin{eqnarray*} d (R^{-1/n}\varphi^{1/n})_{x_n}(v_n)&=&d\Big( \varphi - \frac{1}{n} \varphi_1 \nu+ o(1/n) \Big)_{x_n} (v_n)\\
&=& d\varphi_{x_n}(v_n) -  \frac{1}{n} \Big( (d \varphi_1)_ \nu + \varphi_1 d\nu  \Big)_{x_n} (v_n)+ o(1/n) 
\end{eqnarray*}
Thus there exists a non vanishing $v_0 $ such that a subsequence of $v_n$ tends to $ v_0$ and we obtain
$$ \left\{ \begin{array}{l} d\varphi_x (v_0)=0\\
 (d \varphi_1)_{x}(v_0) \nu(x) +  \varphi_1 (x)d\nu_x(v_0) =0.
\end{array} \right.$$
Remembering that $\varphi_1$ is the first coordinate of $\varphi$, this system implies the vanishing of $\varphi_1 (x) d\nu_x(v_0)$. By Lemma \ref{leg}, $d\nu_x(v_0)$ and $d\varphi_x(v_0)$ cannot vanish simultaneously, therefore $\varphi_1(x)$ vanishes.
Repeating the argument with suitable rotations yields that all the other coordinates of $\varphi(x)$ vanish, a contradiction since $\varphi \in \S^{n+1}.$
\end{proof}

By the previous lemma, since the discussion is local,  we may assume that $\varphi$ is an immersion. The remainder of the proof of Theorem \ref{unbis} follows the lines of Theorem \ref{one},
in particular Lemma \ref{geo} still holds here.


\subsection{The anti de Sitter case}

Let $\bar{\varphi} =(\psi, {\tau}): \M \to Ad\S^{n+2}$ an immersion such that the induced metric
$\bar{g}:= \bar{\varphi}^* \<.,.\>_2$ is spacelike. Let $\bar{ \nu}=( \nu,1),$ be a  normalized, null
vector field which is normal to $\bar{\varphi}.$ We define the  \em null projection \em of $\bar{\varphi}$ to be ${\varphi}:= \psi - {\tau} \nu.$

The fact that $(\nu,1) \in T_{\bar{\varphi}} Ad\S^{n+2},$ i.e.\
$ 0 = \< (\psi, {\tau}),(\nu, 1)\>_2 = \<\psi, \nu\>_1 -\tau$ implies that $\<\psi, \nu\>_1  ={\tau}.$
Hence
\begin{eqnarray*}
\<{\varphi}, {\varphi}\>_1 &=& \<\psi,\psi\>_1 -2 {\tau} \<\psi, \nu\>_1 +  {\tau}^2 \<\nu,\nu\>_1 \\
&=& \<\psi,\psi\>_1 -{\tau}^2\\
&=& \<\bar{\varphi},\bar{\varphi}\>_2\\
&=& -1 ,
\end{eqnarray*}
which shows that $\varphi$ is $\H^{n+1}$-valued.

The remainder of the proof is similar to the previous case (de Sitter case) and is therefore omitted.


\section{The  case of the product of a space form by the real line:  proof of Theorems \ref{two} and \ref{twobis}}

\subsection{The $\S^{n+1} \times \R$ case}

Let $\bar{\varphi} =(\psi, \tau): \M \to \S^{n+1} \times \R$ an immersion such that the induced metric
$\bar{g}:= \bar{\varphi}^*\<.,.\>_1$ is spacelike. Let $\bar{ \nu}=( \nu,1),$ where $\nu \in \S^{n+1},$ be a  normalized, null
normal field along $\bar{\varphi}.$ We set $\varphi:= \cos (\tau) \psi - \sin (\tau) \nu$ and $\nu_{\varphi}=\sin(\tau) \psi + \cos (\tau) \nu.$

\begin{lemm} \label{l1esferico}
 The map $(\varphi, \nu_{\varphi}) : \M \to \S^{n+1} \times \S^{n+1}$ is an immersion. 
\end{lemm}

\begin{lemm} \label{l2esferico}
$$\<d\varphi, \nu_{\varphi}\>_0=0.$$
\end{lemm}

 \begin{lemm}  \label{l3esferico}
Given $x \in \M$ and $ \eps >0,$ there exists a neighbourhood $U$ of $x$ and $t_0 \in (-\eps, \eps)$ such that $\cos(t_0) \varphi + \sin(t_0) \nu_{\varphi}$ is an immersion of $U$, and
 $\cos(t_0) \nu_{\varphi} -\sin(t_0) \varphi$ is its Gauss map. 
  \end{lemm}
 
The proof of Lemmas \ref{l1esferico}, \ref{l2esferico} and \ref{l3esferico} is similar to that of Lemmas \ref{l1mink}, \ref{l2} and \ref{l3mink} of Section \ref{mink} and is therefore ommited. Since we are working locally, Lemma  \ref{l3esferico} proves that, up to a vertical translation, we may assume that $\varphi$ is  an immersion.

\begin{lemm}
Denote by $g= \varphi^*\<.,.\>_0$ the metric induced on $\M$ by $\varphi$ and $A$ the shaped operator associated to $\nu$. Then the metric $\bar{g}=\bar{\varphi}^* \<.,.\>_1$ induced on $\M$ by $\bar{\varphi}$ is given by the formula 
$$
\bar{g}=\cos^{2}(\tau) g(.,.) - 2\sin(\tau)\cos(\tau)g(A.,.) + \sin^{2}(\tau) g(A.,A.).
$$
In particular, the non-degeneracy assumption on $\bar{g}$ implies to $\cot(\tau)$ is not equal to a principal curvature of $\varphi$. Moreover, 
$$
\bar{h}_{\bar{\nu}} := \<\bar{h}(.,.) , \bar{\nu}\>_1=(\cos ^{2}(\tau) - \sin^{2}(\tau))g(A.,.) + \sin(\tau)\cos(\tau)(g(.,.)-g(A.,A.)),
$$
and 
$$
\<\vec{H}_{\bar{\varphi}},\bar{ \nu}\>_1 = \frac{1}{n} \sum_{i=1}^{n} \dfrac{\ka_i+\tan(\tau)}{1 - \tan(\tau)\ka_i},
$$
where the $\ka_i$ are the principal curvatures of $\varphi$.
\end{lemm}

We first claim that if  $\bar{h}_{\bar{\nu}}$
vanishes, it must be totally geodesic: according to the previous lemma, this implies
  $A=\pm \ka Id, $ (hence $\ka$ is constant) and  $\ka = \cot (\tau+\pi/2).$ Then a routine calculation shows that the shape operator
of $\psi=\cos(\tau) \varphi + \sin (\tau)\nu$ vanishes, i.e.\ $\psi$ is totally geodesic. Since the height function $\tau$ is constant, $\bar{\varphi}$ is totally geodesic itself.

We now label the principal curvatures $\kappa_{1}<...< \kappa_{p}$, taking into account their multiplicity $m_i$. Hence $\<\vec{H}_{\bar{\varphi}}, \bar{\nu}\>_1$ vanishes if and only if
\begin{equation*}\label{mt}
\sum_{i=1}^{p} m_i\dfrac{\ka_i +\tan(\tau)}{1-\tan(\tau)\ka_i}=0.
\end{equation*}
Introducing  $s:=\cot(\tau),$ we see that $\bar{\varphi}$ is marginally trapped with respect to $\bar{\nu}$ if and only if the following polynomial vanishes:

\begin{equation*}\label{pmt}
P(s):= \sum_{i=1}^{p}m_i (\kappa_i s+1)\prod_{j\neq i}^{p}(s-\kappa_j)=0.
\end{equation*}
It is easy to check that signs of $P(\kappa_i)$ are alternate. Therefore, the polynomial $P(s)$ admits at least  $p-1$ distinct roots $s_i$ such that $\ka_i< s_i < \ka_{i+1}.$ In particular, the degree of $P(s)$ is at least $p-1.$
Since the term of degree $p$ of $P(s)$ is $ \sum_{i=1}^p m_i \ka_i= nH$, 
it has degree  $p-1$ when $\varphi$ is minimal and
degree $p$ otherwise. In the first case, since we already found $p-1$ roots, we conclude that there are exactly $p-1$ roots.
Observe moreover that if $\varphi$ is minimal and $\tau=0$ (which corresponds to $s= \pm \infty$), the immersion $\bar{\varphi}=(\varphi,0)$ is not only
 marginally trapped, but also minimal.
 In the non-minimal case, by looking at $\lim_{s \to \pm \infty} \frac{P(s)}{s^p},$ we check that there exists one more root $s_p$
 in $(-\infty, \kappa_1)$ or in $( \kappa_p, \infty)$, depending on whether $p$ is even or odd and $H$ is positive or negative.

The conclusion of Theorem \ref{two} comes from the formula
$$\big(\cos(\cot^{-1} s), \sin (\cot^{-1} s) \big)= \left(\frac{s}{\sqrt{1+s^2}} , \frac{1}{\sqrt{1+s^2} }\right).$$
Finally, if $n=2,$ and $\varphi$ is not minimal, setting $a:=\frac{\kappa_1 \kappa_2-1}{\ka_1+\ka_2}$, the polynomial $P(s)$ is equivalent to $s^2 - 2a s-1=0,$
whose two distinct roots are
$s_{\pm}= a\pm \sqrt{a^2 + 1}.$
Hence
$$
\tau_{\pm}= (\cot)^{-1}\left(a \pm \sqrt{a^2 + 1}\right),
$$
so we get the required formula. Observe however that if $a=0, $ then $\tau$ is constant and  moreover  $\psi = \frac{1}{\sqrt{2}}(\pm \varphi  + \nu)$ is minimal, so again $\bar{\varphi}$ is  minimal.

\subsection{The $\H^{n+1} \times \R$ case}
Let $\bar{\varphi} =(\psi, \tau): \M \to \H^{n+1} \times \R$ an immersion whose induced metric
 is spacelike. Let $\bar{ \nu}=( \nu,1),$ where $\nu \in d\S^{n+1},$ be a  normalized, null
normal field along $\bar{\varphi}.$ We set $\varphi:= \cosh (\tau) \psi + \sinh (\tau) \nu.$ 
Reasoning like in the previous cases, we easily prove that, up to a vertical translation and reasoning locally, we may assume that $\varphi$ is an immersion and that $\nu_{\varphi}:=\sinh(\tau) \psi + \cosh (\tau) \nu$ is its Gauss map. Moreover, the non-degeneracy assumption on the induced metric on $\bar{\varphi}$ implies that $\coth(\tau)$ is not equal to a principal curvature $\ka_i$  of $\varphi$. Finally, counting the principal curvatures with their multiplicity $m_i$, we have that  $\<\vec{H}_{\bar{\varphi}},\bar{ \nu}\>_1$ 
vanishes if and only if $\sum_{i=1}^{p} m_i\dfrac{\ka_i-\tanh(\tau)}{1 - \tanh(\tau)\ka_i}$ vanishes as well.
Hence, if $s_i$ is a root of the polynomial 
$$P(s):= \sum_{i=1}^{p}m_i (\kappa_i s-1)\prod_{j\neq i}^{p}(s-\kappa_j)$$
satisfying  in addition $|s_i|>1$, the immersion
\begin{eqnarray*} \bar{\varphi}_i &:=&
\left(\cosh (\coth^{-1} (s_i) )\varphi + \sinh (\coth^{-1} (s_i)) \nu, \coth^{-1} (s_i) \right)\\
&=& \left(  \frac{s_i \varphi + \nu}{\sqrt{s_i^2-1}}, \coth^{-1} (s_i)\right)
\end{eqnarray*}
is marginally trapped.

\medskip

It seems difficult to determine exactly the number $q$ of roots of $P(s)$ such that $|s|>1.$ 
However, observe that 
given  a monotone sequence of $\ka_i$ such that 
$\ka_i<-1,$  $|\ka_i|<1$ or $\ka_i>1,$
the signs of
$P(\ka_i)=m_i (\ka_i^2-1) \prod_{j \neq i}^p (\ka_i -\ka_j)$ are alternate. Hence introducing 
\begin{eqnarray*}
\al&:=&  \# \{ \ka_i,  1 \leq i \leq p, \ka_i <-1\} ,\\
\be &:=&  \# \{ \ka_i,  1 \leq i \leq p,  |\ka_i| <1\}, \\
\ga &:=&  \# \{ \ka_i,  1 \leq i \leq p, \ka_i >1\} ,\\
\delta&:=&  \# \{ \ka_i,  1 \leq i \leq p,  |\ka|=1\} ,
\end{eqnarray*}
We deduce that there exist $\al-1$ roots $s_i$ satisfying $ \ka_i < s_i  <\ka_{i+1} < -1$, giving rise to $\al-1$ marginally trapped immersions
$\bar{\varphi}_i$. Analogously there exist $\ga-1$ solutions satisfying $ 1 <\ka_i <s_i  < \ka_{i+1} .$ Analysing the signs of
$\lim_{s \to \pm \infty} \frac{P(s)}{s^{p}}$ as in the $\S^{n+1} \times \R$ case, we see that  if $\sum_{i=1}^p m_i \ka_i = nH$ does not vanish, the existence of one more solution 
$s \in  (-\infty , \inf_i  {\ka_i} )\cup (\sup_i {\ka_i} , + \infty ) $ is granted. On the other hand, if $\be \neq0, $ the $\be-1$ roots satisfying $-1 < \ka_i <s_i < \ka_{i+1} <1$ lead to no marginally trapped immersion. Finally, if $1$ or $-1$ is a principal curvature, it is also a root of $P(s)$, which again corresponds to no marginally trapped immersion. Finally, if $\varphi$ is not minimal, we obtain the following inequalities:
$$  \al + \ga -1 \leq q \leq p - (\be-1) - \delta=\al+\ga +1   .$$

If $n=2$ and $\varphi$ is not minimal, the polynomial $P(s)$ is equivalent to $s^2 -2as +1, $ where we set $a :=\frac{\ka_1 \ka_2+1}{\ka_1+ \ka_2}.$ Without loss of generality we assume that $a>0.$ If $a<1,$ $P(s)$ has no real solution and if $a=1$, the unique solution is $s=1.$ Finally, if $a>1,$ the two distinct roots of $P(s)$ are $a\pm \sqrt{a^2 - 1},$ one of which is less than one and the other greater than one. 
Hence we get 
$
\tau:= (\cot)^{-1}\left(a + \sqrt{a^2 - 1}\right),
$
so we get the required formula.  Observe that if $|\ka_1| >1 $ and $|\ka_2|>1,$ we have $\al+\ga-1=q=1,$ i.e.\ the left hand side inequality above is sharp.

\section{Examples} \label{examples}

Here we briefly discuss how some of the examples of \cite{CVdV}  can be recovered from our construction.

\medskip

The two families of marginally trapped surfaces found by B.-Y. Chen and J. Van der Veken in $\R^4_1$ are:
$$L_{1}(x,y):=( x,y ,f(x), f(x)),$$ 
where $f$ is an arbitrary differentiable function with $f''(x)$ being nowhere zero, and 
$$L_{2}(x,y):= \left(y\cos x -\int_{0}^{x} r(x)\sin x  dx, y \sin x+\int_{0}^{x} r(x)\cos x  dx, \right.$$ 
$$\left.   q(x)y + \int_0^{x}r(x)q'(x)dx, q(x)y  +  \int_0^{x}r(x)q'(x)dx \right),$$
where $q$ and $r$ are defined on an open interval $I\ni 0$ satisfying $q''(x)+q(x)\neq 0$ for each $x\in I$.

Thus  $L_{1}(x,y)=(\iota(x,y),0) + f(x) (\nu_0,1)$, where $\nu_0=(0,0,1),$
so we are in the first case (null second fundamental form) of  Theorem \ref{one}.  Moreover, since
$$ T(x,y) :=\left(y\cos x -\int_{0}^{x} r(x)\sin x  \, dx, y \sin x+\int_{0}^{x} r(x)\cos x  \, dx \right)$$
is simply a reparametrization of an open subset of the plane, setting
$$\tau(x,y):= q(x)y + \int_0^{x}r(x)q'(x)dx ,$$
the second family takes the form $L_{2}(x,y)=(\iota \circ T (x,y)+ \tau(x,y)\nu_0,\tau(x,y))$,
so we are again in the case of null second fundamental form.

\medskip

Next, consider the immersion in $d\S^4$ given by
$$L_{3}(x,y):=\big(\sin x  \cos y , \sin y, \cos x  \cos y, f(x)\cos y , f(x)\cos y \big)$$ 
where $f$ is an arbitrary differentiable function defined on an open interval $I$ satisfying $f'' + f\neq 0$ at each point in $I.$ It
takes the form
\begin{eqnarray*}
 L_3(x,y) &= &(\sin x  \cos y, \sin y, \cos x  \cos y,0,0) + f(x) \cos y (0,0,0,1,1)\\
&=&(\iota(x,y),0) + \tau(x,y)(\nu,1),\end{eqnarray*}
where  $\tau(x,y):= f(x) \sin y$, $ \nu:=(0,0,0,1)$ and
 $\iota(x,y) : \S^2 \to \S^3$ is the totally geodesic embedding given in coordinates by 
$\iota(x,y)=(\sin x \cos y , \sin y , \cos x \cos y,0) .$  Hence $L_3$ has null second fundamental form (first case of Theorem \ref{unbis}).

\medskip

Finally, consider the immersion in $Ad\S^4$ given by
$$L_4(x,y):=\left( e^y- 2 \sinh y, x e^y, x^2 e^y - \frac{1}{2}e^y,\frac{3}{2}e^y - 2 \sinh y,x^2 e^y \right).$$
A straightgforward calculation shows that a normalized, null normal vector along $L_4$ is $\bar{\nu}=(-1,0,1,-1,1)=(\nu,1)$. Since $\bar{\nu}$ is constant,
$\<\bar{h}(.,.), \bar{\nu}\>_2$ vanishes, so is in particular the second fundamental form $\bar{h}(.,.)$ is null and $L_4$ is marginally trapped. Observe moreover that
$\varphi:=\psi - \tau \nu =\psi - x^2 e^y (-1,0,1,-1)$ is an immersion whose normal unit vector $\nu=(-1,0,1,-1)$ is constant, therefore $\varphi$ is totally geodesic.

\medskip 

We leave to the reader the easy task to check that all other examples of \cite{CVdV} have null second fundamental form.


\break
 
 \noindent

\begin{multicols}{2}

\noindent   Henri Anciaux \\
Universit\'e Libre de Bruxelles \\
  CP 216, local O.7.110  \\
   Bd du Triomphe \\
 1050 Brussels, Belgium   \\
henri.anciaux@gmail.com \\

\columnbreak

\noindent Yamile Godoy \\ 
  FaMAF-CIEM, \\Ciudad Universitaria, \\
5000 C\' ordoba, Argentina 
\\
 yamile.godoy@gmail.com
\end{multicols}

\end{document}